\documentclass[12pt]{article}
\usepackage[usenames]{color}
\usepackage{latexsym}
\usepackage{amsmath,amsthm,amssymb}

\newcommand{\f}{\ensuremath \mathrm{FPP}}
\newcommand{\ord}{\ensuremath \mathop \mathrm{ord} \nolimits}

 \advance\textwidth  by 5.0cm
 \advance\textheight by 6cm

\newtheorem{theorem}{Theorem}[section]

\newtheorem{proposition}[theorem]{Proposition}
\newtheorem{corollary}[theorem]{Corollary}

\theoremstyle{definition}
\newtheorem{definition}[theorem]{Definition}
\newtheorem{example}[theorem]{Example}

\theoremstyle{remark}
\newtheorem{remark}[theorem]{Remark}
\theoremstyle{property}

\renewenvironment{proof}{\noindent{\it Proof}.}{\qed}

\title{Counterexamples for Frobenius primality test}

\author{Sergey Khashin\\
        Department of Mathematics, Ivanovo State University, Ivanovo, Russia \\
        \texttt{khash2@gmail.com}\\
        http://math.ivanovo.ac.ru/dalgebra/Khashin/index.html }
\begin{document}
\maketitle


\begin{abstract}
At present one can not find a single counterexample to even a simplest version
of Frobenius primality test. The assessment of probability of the mistake,
presented in \cite{DamFr} is strongly overestimated. In the present paper, the
properties of simple divisors of FPP-numbers are proved (Theorems \ref{thm1},
\ref{thm26}). The lower bound for FPP are given in the proposition
(\ref{prop31}).

\end{abstract}

\section{Introduction}
The most powerful elementary probabilistic method for primality test is the
Frobenius test~\cite{GrPom,Gr,DamFr}. Frobenius pseudoprimes are the natural
numbers for which this test fails. There are several slightly different
definitions of Frobenius pseudoprimes (FPP), which are almost equivalent. The
one that we use in the present paper is the following.
\begin{definition}\label{FrobTest}
Frobenius pseudoprime (FPP) is a composite odd integer $n$ such that it is not a perfect square provided
\[
{(1+\sqrt{c})^n} \equiv (1-\sqrt{c}) \mod n \ ,
\]
where $c$ is the smallest odd prime number with the Jacobi symbol $J(c/n)=-1$.
\end{definition}

\begin{example}
Let $n=7$. Then $J(3/n)=-1$, $c=3$ and $(1+\sqrt{3})^n = 568+328 \sqrt{3} \equiv 1 -
\sqrt{3} \mod n$. This means that $7$ is a Frobenius prime, and not a FPP.
\end{example}
\begin{definition}
Frobenius pseudoprime $\f(a,b,c)$ with parameters $(a,b,c)$ is a composite odd integer $n$ such that it
is not a perfect square provided
\begin{equation} \label{eq:abc}
{(a+b\sqrt{c})^n} \equiv (a-b\sqrt{c}) \mod n \ ,
\end{equation}
where $a,b,c$ are co-prime with $n$, $c$ is free of squares and
$J(c/n)=-1$.
\end{definition}

\begin{remark}
Condition~\eqref{eq:abc} can be re-written as follows:
\[
(a+b\sqrt{c})^{n+1} \equiv (a^2-b^2c) \mod n \ .
\]
\end{remark}
\begin{remark} If $n$ is an $\f(a,b,c)$, then $n$ is a pseudo-prime with the base $N(z) = (a^2-b^2c)
\mod n$.
\end{remark}
Some improvements of the Frobenius test are suggested in~\cite{DamFr, Seysen}. They are aimed to lessen
the percentage of errors, that is the number of the corresponding $\f$. However, there are no
known examples of $\f$ numbers that fail even the initial version of the Frobenius test.
\begin{remark}
We note that an example suggested in the book~\cite{GrPom} does not work as an example of such an $\f$, at least
for our definition of an $\f$. Indeed, the suggested example is number $5777=53{\cdot}109$, which is supposed to be an
FPP with $a=b=1$, $c=5$. However,
\[
{(1+\sqrt{5})^{5778}}\equiv {5342 + 0\cdot \sqrt{5}} \mod 5777 \ .
\]
Thus, the irrational part of the number equals to $0$, as it should be, while the rational one
is not equal $1-5=-4$.
\end{remark}
\begin{remark}
In~\cite{DamFr} an estimate for the probability of an error in the Frobenius test is obtained. Notice that
the ``liars`` presented in Theorem 8 and Lemma 7 of this paper cannot serve as examples of
$\f(a,b,c)$.
Given number $n$ for which we use the Frobenius test, the ``liars`` are defined as numbers of the form $z=a+b\sqrt{c}$
such that is $z^n=\overline{z}$. For example, for $c=3$ and $n=5\cdot 7 \cdot
17$ one can find $200$ ``liars``. However, among them there is no single ``liar`` that has both of the corresponding $a$ and $b$ coprime with
$n$. In general, one can see that there are a lot of ``liars`` that are not $\f$, and therefore,
the estimate given in~\cite{DamFr} may be improved.
\end{remark}

In the present paper we show that there are no examples of $\f$ that are less than $2^{60}$.
The paper is organized as follows. In Sec.~\ref{sec:main} some $\f$ properties are proved. In Sec.~\ref{sec:num} we use the obtained results to conduct some numerical experiments.
The results of the computations based on Theorem~\ref{thm1} are stated in Proposition~\ref{prop31}.
Precisely, we obtain the lower bound for the product of all prime factors but one for an FPP.
The results of the computations based on Theorem~\ref{thm26} are stated in Proposition~\ref{pr32}. Theorem~\ref{thm26} states
some properties of the prime factors with condition $J(c/p)=+1$. Proposition~\ref{pr32} states that
all such factors are greater than $2^{30}$.
Our main result  is Proposition~\ref{main_p}, which states the lower bound for an FPP.

\section{Main results}
\label{sec:main}
Let $c$ be a square-free integer and $z=a+b\sqrt{c} \in {\Bbb
Z}[\sqrt{c}\,]$. We shall say that $a=Rat(z)$ is the ``rational part'', and $b\sqrt{c}=Irr(z)$ is the``irrational part`` of $z$. We shall say that integer
$N(z)=a^2-b^2 c$ is the norm of $z$, and $\overline{z} = a-b\sqrt{c}$ is the conjugated to $z$ number.
Then $N(z_1z_2)=N(z_1)N(z_2)$, $N(z) = z \cdot \overline{z}$.

Let $p$ be an odd prime and $c$ is not a square modulo $p$, i.e. $J(c/p)=-1$. Then ring ${\Bbb Z}_p[\sqrt{c}\ ]$ is isomorphic to
Galois field $GF(p^2)$. If $J(c/p)=+1$, then there exist $d \in {\Bbb Z}_p, d^2=c$ and
there is an isomorphism between ${\Bbb Z}_p[\sqrt{c}\,]$ and
${\Bbb Z}_p \times {\Bbb Z}_p$,
\begin{equation} \label{eq_a1}
{\Bbb Z}_p[\sqrt{c}\,] \to {\Bbb Z}_p \times {\Bbb Z}_p: \quad  a+b\sqrt{c} \mapsto (a+bd, a-bd) \ .
\end{equation}

By $\ord(z,k)$ we denote  the order of $z=a+b\sqrt{c}$ in the multiplicative group ${\Bbb Z}_k[\sqrt{c}\,]^*$.

The following statement (in a slightly different formulation) is proved in~\cite{DamFr, Seysen}.
\begin{theorem} \label{thm1}
Let $n$ be an $\f$ with parameters $(a,b,c)$, $n=pq$, where $p$ -- prime, $z=a+b\sqrt{c}$.
Then

a) if $J(c/p)=-1$, then $z^{q-1} \equiv 1 \mod p$.

b) if $J(c/p)=+1$, then $z^{q+1} \equiv N(z) \mod p$.

\end{theorem}
\begin{proof} \\
a) In this case the order of $z \in {\Bbb Z}_p[\sqrt{c}\,]$ is a divisor of $p^2-1$, that is the order is co-prime with $p$.
Since $z^{pq} \equiv \overline{z} \mod n$, then  $z^{pq} \equiv \overline{z} \mod p$ and hence $z^{p} \equiv \overline{z} \mod p$. Therefore, $z^{pq} \equiv z^p \mod p$ or $z^{p(q-1)} \equiv 1 \mod p$.
Since the order of $z$ is co-prime with $p$, we conclude that $z^{q-1} \equiv 1 \mod p$.

\noindent b) In this case $z^{pq+1}\equiv N(z) \mod n$, and, therefore, $z^{pq+1}\equiv N(z) \mod p$.
Since $J(c/p)=+1$, we can conclude that $z^{p-1}\equiv 1
\mod p$. Hence,
$
 z^{pq+1} = z^{(p-1)q+q+1} \equiv z^{q+1} \equiv N(z) \mod p.
$
\end{proof}
\begin{corollary} Let $n$ be an $\f$ with parameters $(a,b,c)$ and let $n=pq$, where $p$ be a prime and $J(c/p)=-1$.
Then $q$ is co-prime with $\ord(a+b\sqrt{c}, p)$.
\end{corollary}

\begin{corollary} \label{cor22}
Let $z=a+b\sqrt{c} \in {\Bbb Z}$ and $z^k=a_k+b_k \sqrt{c}$. Let $n$ be an $\f$ with parameters $(a,b,c)$,
and $n=pq$, where $p$ be a prime. Then

a) if $J(c/p)=-1$ (and $J(c/q)=+1$), then $p$ is a prime factor of $\gcd(a_{q-1}-1,
b_{q-1})$.

b) if $J(c/p)=+1$ (and $J(c/q)=-1$), then $p$ is a prime factor of
$\gcd(a_{q+1}-(a^2-b^2c), b_{q+1})$.

\end{corollary}

\begin{example}
Let $q=31,\ c=5$, then $J(c/q)=+1$ and
\[
 (1+\sqrt{c})^{q-1} = a_{q-1}+  b_{q-1}\sqrt{c}
 =998847258034176+446698073620480 \sqrt{c}.
\]
Since $gcd(a_{q-1}-1, b_{q-1})=104005$, then $p$ is a prime factor of $104005$, that is $p$ is one of the following numbers: $5, 11, 31, 61$.
\end{example}

\begin{example}
Let $q=37,\ c=5$, then $J(c/q)=-1$ and
\[
 (1+\sqrt{c})^{q+1} = a_{q+1}+  b_{q+1}\sqrt{c}
 =12012687213792854016+5372237040496672768 \sqrt{c}.
\]
Since $gcd(a_{q+1}-(1-c), b_{q+1})=148$, then $p$ is a prime factor of $148$, that is $p$ is one of the following numbers: $2, 37$.
\end{example}

\begin{remark} Although numbers $a_k, b_k$ grow fast, it appears that the corresponding $\gcd$s (from Corrolary~\ref{cor22}) do not
grow at as fast and can be factorized until the corresponding $q$ is less then some hundreds of thousands.
\end{remark}

\begin{theorem}\label{thm26}
Let $n$ be an $\f$ with parameters $(a,b,c)$, and $p$ be a prime divisor of $n$ and $J(c/p)=+1$.
Then
\begin{equation}
  \gcd( \ord(w_1 w_2, p), \ord(w_1/w_2, p)) \le 2\,,
\end{equation}
and
\begin{equation} \label{eq_a3}
 w_1^q \equiv w_2, \quad w_2^q \equiv w_1 \mod p
\end{equation}
where $w_1 = a+bd,\ w_2=a-bd \in {\Bbb Z}_p$, $d^2=c$ and $q=n/p$.
\end{theorem}
\begin{proof} Since $J(c/p)=+1$, there exists such $d \in {\Bbb Z}_p$ that $d^2=c$.
According to Theorem~\ref{thm1}, $z^{q+1} \equiv N(z)$ in the ring ${\Bbb
Z}_p[\sqrt{c}\,]$, where $z=a+b\sqrt{c}$, $q=n/p$.
Since $N(z) = z\overline{z}$, we have that
\begin{equation} \label{eq_a2}
  z^q \equiv \overline{z}.
\end{equation}
Using isomorphism \eqref{eq_a1} we obtain \eqref{eq_a3}, where $w_1 = a+bd$, $w_2 = a-bd$.
If we denote $\gamma = w_1/w_2$, $\delta = w_1 w_2$, then:
\begin{equation}
 \gamma^{q+1} \equiv 1, \quad \delta^{q-1} \equiv 1 \mod p.
\end{equation}
Therefore,
\begin{equation} \label{eq_a5}
 q \equiv -1 \mod(\ord(\gamma, p)) , \quad  q \equiv 1 \mod(\ord(\delta, p)) \ .
\end{equation}
Notice that equalities \eqref{eq_a5} can not be satisfied simultaneously in the case when $\ord(\gamma, p), \ord(\delta, p)$
has common divisor greater than $2$.
\end{proof}

Primes $p$ from Theorem~\ref{thm26} are very rare. For example, for $z=1+\sqrt{3}$, the minimal $p$ is greater
then $P_0=2{\cdot}10^9$. For $z=1+\sqrt{5}$ there are only two primes $p$,
$p=61681$ and $p=363\,101\,449$ that are smaller than $P_0$ and satisfy the conditions of Theorem~\ref{thm26}.
For $z=1+\sqrt{7}$ there are only two such primes $p$, $p=31$ and $p=3923$.
In general, even if we count for different $c<128$ all those primes $p$~\footnote{That is all those $p$ that satisfy $p<P_0$ and the conditions of Theorem~\ref{thm26}.} together, then
there are only $99$ of those.


\begin{remark} \label{rem27}
Conditions (\ref{eq_a5}) means
\begin{equation}
 q \equiv q_0 \mod LCM( ord(\gamma,p), ord(\delta, p))
\end{equation}
for some $q_0$. Since both $\gamma$, и $\delta$ are the powers of $w_1$, then the
congruences (\ref{eq_a3}) depends only on $q \mod ord(w_1,p)$.
\end{remark}



\begin{theorem} \label{thm_29}
Let $n$ be an $\f$ with parameters $(a,b,c)$, $p$ be a prime divisor
of $n$, $n=p^sq$, $s\ge 1$ and  $J(c/p)=-1$. Then
\[
 z^p \equiv \overline{z} \mod p^s.
\]
\end{theorem}

{\bf Proof}. If $s=0$, as $J(c/p)=-1$, then $z^p \equiv \overline{z}
\mod p$. Now assume $s>1$. Since $z^n \equiv \overline{z} \mod n$,
so ${z^{n^2-1} \equiv 1 \mod n}$, that is $ord(z,n)$ is the divider
of $n^2-1$. Hence this order is coprime with $n$ and, as a result,
with $p$. So $ord(z,p^s)$ is also coprime with $p$. On the other
hand, $ord(z,p^s) = ord(z,p)p^t$ for some $t\ge 0$. As $ord(z,p^s)$
coprime with $p$, we obtain:
\begin{equation}\label{eq_10}
 ord(z,p^s) = ord(z,p).
\end{equation}

From (\ref{eq_10}) follows that $N(z^{p}) \equiv 1 \mod p^s$ ($\Rightarrow N(z^{p+1})
\equiv N(z))$ and $(z/\overline{z})^{p+1} \equiv 1 \mod p^s$. Hence, the number $z^{p+1}$
have zero irrational part and the norm $N(z)^2$. So
\[
z^{p+1} \equiv \pm N(z) \mod p^s.
\]
As $z^{p+1} \equiv N(z) \mod p$, so $z^{p+1} \equiv N(z) \mod p^s$. $\blacksquare$

\section{Numerical results}
\label{sec:num}
\begin{proposition}\label{prop31}
Let $n$ be an $\f$ and $p$ be a prime divisor. Then $n/p>2^{17}$.
\end{proposition}

{\bf Proof}. Direct computation, based on corollary \ref{cor22}.
$\blacksquare$

\begin{proposition}  \label{pr32}
Let $p$ be a  prime factor of some $\f$ with $c<128$ and
$J(c/p)=+1$. Then $p > 2^{30}$.
\end{proposition}

{\bf Proof}. Direct computation, based on Theorem \ref{thm26}. $\blacksquare$

\begin{proposition}  \label{pr33}
Let $p$ be a prime factor of some $\f$ with $c<128$ and $J(c/p)=-1$.
Then ${p > 1\,663\,000\,000}$.
\end{proposition}

{\bf Proof}. Direct computation, based on Theorem \ref{thm_29}. $\blacksquare$

\begin{corollary}
Let $p$ be an $\f$ with multiple prime factor and $c<128$. Then $n
> 2^{60}$.
\end{corollary}

\begin{proposition}
Let  $n<2^{60}$ be and $\f$ with $c<128$, $p$ be its prime factor
and $J(c/p)=-1$. Then $p$ can't lie (be?) in the range of $\ 46\,000
< p < 2^{30}$.
\end{proposition}

{\bf Proof}. Theorem (\ref{thm1}) shows that $(1+\sqrt{c})^{q-1}
\equiv 1 \mod p$, where $q=n/p$. That is $q \equiv 1 \mod r$, where
$r$ is an order of $(1+\sqrt{c})$ in multiplicative group ${\Bbb
Z}_p[\sqrt{c}\,]^*$. As $p\cdot q<2^{60}$, that for all such pairs
$(p,c)$ from these intervals one can test all possible
values.$\blacksquare$

\begin{proposition}
Let $n<2^{60}$ be an $\f$ with $c<128$, $p_1 > p_2>17$ are two its
prime divisors and $J(c/p_1)=J(c/p_2)=-1$. Then $p_1$ can't lie
(be?) in the range of $\ 5419 < p_1 < 46\,000$.
\end{proposition}

{\bf Proof}. Let $n=p_1p_2q$ and $z=1+\sqrt{c}$. From theorem (\ref{thm1}) in
this case:
\[
  z^{p_2q-1} \equiv 1 \mod p_1, \quad  z^{p_1q-1} \equiv 1 \mod p_2,
\]
or
\[
  p_2q \equiv 1 \mod r_1, \quad  p_1q \equiv 1 \mod r_2,
\]
where $r_i$ -- order of $z$ in multiplicative group ${\Bbb
Z}_{p_i}[\sqrt{c}\,]^*$, that is
\[
  q \equiv 1/p_2 \mod r_1, \quad  q \equiv 1/p_1 \mod r_2.
\]

By Chinese remainder theorem both this congruences (if they are
compatible) are equivalent to one:
\[
  q \equiv q_0 \mod lcm(r_1, r_2).
\]
As $p_1p_2 q<2^{60}$, then for all triples $(p_1,p_2,c)$ from this intervals we can check
all possible values of $q$. $\blacksquare$

\begin{corollary}
Let $n<2^{60}$ be an $\f$ with $c<128$ and without divisors $<17$. Then $n$
have not divisors in interval $5419<p<2^{30}$.
\end{corollary}

The following statements are proved similarly.

\begin{proposition}
Let $n<2^{60}$ be an $\f$ with $c<128$ and  $p_1 > p_2> p_3>13$ are a three
prime factor with $J(c/p_1)=J(c/p_2)=J(c/p_3)=-1$. Then $p_1$  can't lie in the
range $433<p_1 \le 5419$.
\end{proposition}

\begin{corollary}
Let $n<2^{60}$  be an $\f$  with $c<128$ and without divisors$<17$. Then $n$
does not have prime divisors $p$ in interval $433<p<2^{30}$.
\end{corollary}

\begin{proposition}
Let $n<2^{60}$  be an $\f$  with $c<128$ and $p_1
> p_2> p_3>17$ are three prime divisors, $J(c/p_1)=J(c/p_2)=J(c/p_3)=-1$. Then
$p_1$ does not be in interval $29 \le p_1 \le 433$.
\end{proposition}

\begin{corollary}
Let $n<2^{60}$  be an $\f$  with $c<128$ and without divisors $<17$. Then $n$
does not have divisors less then $2^{30}$.
\end{corollary}

And, at last, from here follows

\begin{proposition}\label{main_p}
There are no $\f$ less then $2^{60}$ with $c<128$ and without divisors $\le
17$.
\end{proposition}

\section*{Conclusions}

We remark that in our computations we assume that
$c<128$, which means that $n$ has a quadratic non-residue in the interval  $2\dots 127$.
Numbers $n$ with $c \geq 128$ occur with a very small probability, roughly $2^{-32}$.
One of the first numbers that our computations do not include is rather large number $n=196\,265\,095\,009$ (prime) with
$c=131$.

This bound  can be improved by some orders.
Besides that an appropriate modification of
these algorithms would help to get rid from the restrictions on the parameter $c$, $128$
in current paper.

\bibliographystyle{plain}

\end{document}